\documentclass[12pt]{amsart}

\usepackage{amsmath}
\usepackage{amsfonts}
\usepackage{amsthm}

\usepackage[cmtip,arrow]{xy}
\usepackage{pb-diagram, pb-xy}

\newtheorem{theorem}{Theorem}
\newtheorem{corollary}{Corollary}
\newtheorem{lemma}{Lemma}
\newtheorem{proposition}{Proposition}

\newtheorem{example}{Example}
\newtheorem{remark}{Remark}

\renewcommand{\P}{\mathcal{P}}

\begin{document}
\title{Projectively coresolved Gorenstein flat and Ding projective modules}
\author{Alina Iacob}
\address[A. Iacob]{Department of Mathematical Sciences. Georgia Southern University. Statesboro (GA) 30460-8093. USA}
\email{aiacob@GeorgiaSouthern.edu}

\begin{abstract}
We give necessary and sufficient conditions in order for the class of projectively coresolved Gorenstein flat modules, $\mathcal{PGF}$, (respectively that of projectively coresolved Gorenstein $\mathcal{B}$ flat modules, $\mathcal{PGF}_{\mathcal{B}}$) to coincide with the class of Ding projective modules ($\mathcal{DP})$. We show that $\mathcal{PGF} = \mathcal{DP}$ if and only if every Ding projective module is Gorenstein flat. This is the case if the ring $R$ is coherent for example. We include an example to show that the coherence is a sufficient, but not a necessary condition in order to have $\mathcal{PGF} = \mathcal{DP}$. We also show that $\mathcal{PGF} = \mathcal{DP}$ over any ring $R$ of finite weak Gorenstein global dimension (this condition is also sufficient, but not necessary). We prove that if the class of Ding projective modules, $\mathcal{DP}$, is covering then the ring $R$ is perfect. And we show that, over a coherent ring $R$, the converse also holds. We also give necessary and sufficient conditions in order to have $\mathcal{PGF} = \mathcal{GP}$, where $\mathcal{GP}$ is the class of Gorenstein projective modules.
\end{abstract}
\maketitle




\section{introduction}
We consider several classes of modules:\\
1. The \emph{Gorenstein projective} modules were introduced in 1995 (\cite{enochs:95:gorenstein}) by Enochs and Jenda, as a generalization of Auslander's modules of G-dimension zero. They are the cycles of exact complexes of projective modules that stay exact when applying a functor $Hom(-,P)$ for any projective module $P$. We use $\mathcal{GP}$ to denote the class of Gorenstein projective modules.\\
2, Another generalization of Auslander's modules of G-dimension zero is the class of \emph{Gorenstein flat} modules. They were introduced in 1994 (\cite{enochs:94:gorflat}) by Enochs, Jenda and Torrecillas. They are the cycles of the exact complexes of flat modules that remain exact when tensored with any injective module. We use $\mathcal{GF}$ to denote the class of Gorenstein flat modules. It is known that over certain classes of rings any Gorenstein projective module is also Gorenstein flat. For example, this is the case for any right coherent and left n-perfect ring. But whether or not this inclusion holds in general, this is an open question\\
3. Let $\mathcal{B}$ be any class of right $R$-modules. The \emph{Gorenstein $\mathcal{B}$ flat} modules were defined in \cite{EIP} as a relative version of the Gorenstein flat modules. They are the cycles of exact complexes of flat modules that stay exact when tensored with any module $B \in \mathcal{B}$. We denote this class of modules by $\mathcal{GF}_{\mathcal{B}}$. It is immediate from the definition that $\mathcal{GF}_{\mathcal{B}} \subseteq \mathcal{GF}$ when $\mathcal{B}$ contains the injective modules.\\
4. The \emph{Ding projective} modules were introduced by Ding, Li and Mao in \cite{ding:09:ding.proj} where they were called strongly
Gorenstein flat modules. They are the cycles of the exact complexes of projective modules that stay exact when applying a functor $Hom(-,F)$, for any flat module $F$. In \cite{gil:10:model}, \cite{gil:17:ding} Gillespie renamed these modules Ding projective modules. We use $\mathcal{DP}$ to denote this class of modules. It is immediate from the definition that any Ding projective module is Gorenstein projective.\\
5. Recently Saroch and Stovicek introduced a new class of modules (in \cite{saroch:18:gor}) - the \emph{projectively coresolved Gorenstein flat} modules, or $PGF$-modules for short. They are the cycles of exact complexes of projective modules that remain exact when tensored with any injective module. The class of projectively coresolved Gorenstein flat modules is denoted $\mathcal{PGF}$. As noted in \cite{saroch:18:gor}, these modules are both Gorenstein projective and Gorenstein flat. Also, any projectively coresolved Gorenstein flat module is Ding projective (see section 2). \\
6. Let $\mathcal{B}$ be any class of right $R$-modules. The \emph{projectively coresolved Gorenstein $\mathcal{B}$ flat} modules ($\mathcal{PGF}_{\mathcal{B}}$ modules for short) were introduced in \cite{EIP}; they are  the cycles of the exact complexes of projective modules that remain exact when tensored with any module $B \in \mathcal{B}$.\\
7. The \emph{Gorenstein AC-projective} modules were introduced by Bravo, Gillespie, and Hovey in \cite{gillespie:14:stable}. We recall first that a right $R$-module $M$ is said to be of type $FP_{\infty}$ if there exists an exact complex $ \ldots \rightarrow P_1 \rightarrow P_0 \rightarrow M \rightarrow 0$ with all the $P_j$ finitely generated projective modules. We also recall that a left $R$-module $N$ is called \emph{level} if $Tor_1(M, N) = 0$ for every right R-module M of type $FP_{\infty}$. The Gorenstein AC-projective modules are the cycles of the exact complexes of projective modules that stay exact when applying any functor $Hom(-,L)$ where $L$ is any level module. We use the notation $\mathcal{GF}_{ac}$ for the class of Gorenstein AC-projective modules.

We consider the following questions:\\
- When is it true that the class of $\mathcal{PGF}_{\mathcal{B}}$ modules coincides with that of Ding projective modules? In particular when is it true that the $\mathcal{PGF}$ modules are the same with the Ding projective modules?\\
- When is it true that the Ding projective modules and the Gorenstein projective modules coincide?\\
- When is it true that the class of $\mathcal{PGF}$ modules coincides with that of Gorenstein projective modules? 

Since some of our results involve semi-definable classes, we recall below a few more definitions.

We recall first that a class of modules $\mathcal{D}$ is called \emph{definable} if it is closed under direct products, direct limits and pure submodules. It is known that such a definable class $\mathcal{D}$ has an elementary cogenerator. A module $D_0 \in \mathcal{D}$ is said to be an \emph{elementary cogenerator} for $\mathcal{D}$ if it is a pure injective module (i.e. it is injective with respect to pure exact sequences), and if every $D \in \mathcal{D}$ is a pure submodule of some direct product of copies of $D_0$.\\
We are particularly interested in classes of modules that contain an elementary cogenerator of their definable closure. The \emph{definable closure} of a class of modules $\mathcal{B}$, denoted $<\mathcal{B}>$, is the smallest definable class containing $\mathcal{B}$.\\
By \cite{EIP}, a class $\mathcal{B}$ is called \emph{semi-definable} if it is closed under arbitrary direct products and contains an elementary cogenerator of its definable closure $<\mathcal{B}>$.


We prove first that if $\mathcal{B}$ is a semi-definable class of right $R$-modules such that $\mathcal{B}$ contains the class of injectives then $\mathcal{PGF}_{\mathcal{B}} = \mathcal{DP} \bigcap \mathcal{GF}_{\mathcal{B}}$. Consequently $\mathcal{PGF}_{\mathcal{B}} = \mathcal{DP}$ if and only if $\mathcal{DP} \subseteq \mathcal{GF}_{\mathcal{B}}$. 
Then we show that when $\mathcal{B}$ is a semi-definable class of right $R$-modules that contains the injective modules we have that $\mathcal{DP} = \mathcal{PGF}_{\mathcal{B}}$ if and only if every Ding projective module has finite Gorenstein $\mathcal{B}$ flat dimension. In particular, over any ring $R$, we have that $\mathcal{DP} = \mathcal{PGF}$ if and only if every Ding projective module has finite Gorenstein flat dimension. Consequently we have that $\mathcal{PGF} = \mathcal{DP}$ over any ring $R$ of finite weak Gorenstein global dimension. The condition that the ring has finite weak Gorenstein global dimension is a sufficient, but not a necessary one for the two classes to coincide. Theorem 1 shows that $\mathcal{DP} = \mathcal{PGF}$ if and only if the class $Inj^+$ of all character modules of injective right $R$-modules is contained in $\mathcal{DP}^\bot$. This implies that $\mathcal{DP} = \mathcal{PGF}$ over any coherent ring $R$ (not necessarily of finite weak Gorenstein global dimension). Example 1 shows that the coherence is also a sufficient but not a necessary condition on the ring: if $R$ is a noncoherent ring of finite global dimension, then we have $\mathcal{DP} = \mathcal{GP} = \mathcal{PGF}$.\\
We prove that if the class of Ding projective modules, $\mathcal{DP}$, is covering then the ring $R$ is perfect. And we show that, over a coherent ring $R$, the converse also holds.\\ We also give necessary and sufficient conditions in order to have $\mathcal{PGF} = \mathcal{GP}$. In particular, we show (Theorem 4) that over a coherent ring $R$ we have that $\mathcal{PGF} = \mathcal{GP}$ if and only if $Flat \subseteq \mathcal{GP}^\bot$. We also prove (Proposition 9) that if $R$ is a ring such that every injective module has finite flat dimension then we have $\mathcal{GP} = \mathcal{DP} = \mathcal{PGF}$.

\section{results}

Throughout, $R$ denotes an associative ring with unity. Unless otherwise specified, by $R$-module we mean left $R$-module. We use $Proj$, $Flat$, and $Inj$ to denote the classes of projective, flat, and, respectively, injective modules. $\mathcal{B}$ denotes a (fixed) class of right $R$-modules.

Given a class of modules $\mathcal{C}$, we denote by $\mathcal{C}^\bot$ its right orthogonal class, i.e. the class of modules $X$ such that $Ext^1(C,X)=0$ for any $C \in \mathcal{C}$. The left orthogonal class of $\mathcal{C}$ is defined dually. We recall that a pair of classes of $R$-modules $(\mathcal{C}, \mathcal{L})$, is a \emph{cotorsion pair} if $\mathcal{C} ^ \bot = \mathcal{L}$ and $^ \bot \mathcal{L} = \mathcal{C}$. A cotorsion pair is \emph{complete} if for any $_RM$ there are exact sequences $0 \rightarrow L \rightarrow C \rightarrow M \rightarrow 0$ and respectively $0 \rightarrow M \rightarrow L' \rightarrow C' \rightarrow 0$ with $C, C' \in \mathcal{C}$ and $L, L' \in \mathcal{L}$. We also recall that a cotorsion pair is said to be \emph{hereditary} if $Ext^i(C,L)=0$ for all $i \ge 1$, all $C \in \mathcal{C}$, all $L \in \mathcal{L}$. It is known that this is equivalent with the class $\mathcal{C}$ being closed under kernels of epimorphisms, and it is also equivalent with the condition that $\mathcal{L}$ is closed under cokernels of monomorphisms.

Over any ring $R$, we have that $\mathcal{PGF} \subseteq \mathcal{GF}$ (by definition). By \cite{saroch:18:gor} we have $\mathcal{GP}_{ac} \subseteq \mathcal{PGF} \subseteq \mathcal{GP}$ over any ring $R$ (see \cite{saroch:18:gor}, page 15 and Theorem 3.4).

\begin{lemma}
If $\mathcal{B} \supseteq Inj$ then $\mathcal{PGF}_{\mathcal{B}} \subseteq \mathcal{DP}$.
\end{lemma}

\begin{proof}
If $\mathcal{B} \supseteq Inj$ then $Flat \subseteq \mathcal{PGF}_{\mathcal{B}} ^ \bot$ (\cite{EIP}, Proposition 2.9(2)). Let $M \in \mathcal{PGF}_{\mathcal{B}}$. Then there exists an exact complex of projective modules $P = \ldots \rightarrow P_1 \rightarrow P_0 \rightarrow P_{-1} \rightarrow \ldots$ such that $Z_j P \in \mathcal{PGF}_{\mathcal{B}}$ for all $j$. Then $Ext^1(Z_jP, F) =0$ for all $j$, for any flat module $F$, so $Hom(P, F)$ is exact for any flat module $F$. It follows that $M \in \mathcal{DP}$.
\end{proof}

\begin{corollary}
 $\mathcal{PGF} \subseteq \mathcal{DP}$ over any ring $R$.
 \end{corollary}

\begin{proof} By Lemma 1, for $\mathcal{B}$ being the class of right injective modules.
\end{proof}

Thus we have that over any ring $R$, $\mathcal{GP}_{ac} \subseteq \mathcal{PGF} \subseteq \mathcal{DP} \subseteq \mathcal{GP}$. So, if $\mathcal{PGF} = \mathcal{GP}$, then we have that $\mathcal{PGF} = \mathcal{DP} = \mathcal{GP}$.



\begin{lemma}
If $\mathcal{B}$ is semi-definable and $\mathcal{B} \supseteq Inj$ then $\mathcal{PGF}_{\mathcal{B}}  = \mathcal{DP} \bigcap \mathcal{GF}_{\mathcal{B}}$.
\end{lemma}

\begin{proof}
"$\subseteq$" Since $\mathcal{PGF}_{\mathcal{B}} \subseteq \mathcal{DP}$ in this case and since $\mathcal{PGF}_{\mathcal{B}} \subseteq \mathcal{GF}_{\mathcal{B}}$ (by definition, over any ring and for any class of right $R$-modules $\mathcal{B}$) it follows that $\mathcal{PGF}_{\mathcal{B}}  \subseteq \mathcal{DP} \bigcap \mathcal{GF}_{\mathcal{B}}$.\\
"$\supseteq$" Let $M \in \mathcal{DP} \bigcap \mathcal{GF}_{\mathcal{B}}$.  Since $\mathcal{B}$ is semi-definable we have that $(\mathcal{PGF}_{\mathcal{B}} , \mathcal{PGF}_{\mathcal{B}}^\perp)$ is a complete cotorsion pair (by \cite{EIP}, Theorem 2.13). Thus there exists a short exact sequence $0 \rightarrow A \rightarrow D \rightarrow M \rightarrow 0$ with $D \in \mathcal{PGF}_{\mathcal{B}}$, and $A \in \mathcal{PGF}_{\mathcal{B}}^\perp$. By \cite{EIP}, Corollary 2.20, $\mathcal{GF}_{\mathcal{B}}$ is closed under kernels of epimorphisms when $\mathcal{B}$ is semi definable. Since $M \in \mathcal{GF}_{\mathcal{B}}$, and $D \in \mathcal{PGF}_{\mathcal{B}} \subseteq \mathcal{GF}_{\mathcal{B}}$, it follows that $A \in \mathcal{GF}_{\mathcal{B}}$.
So $A \in \mathcal{GF}_{\mathcal{B}} \bigcap \mathcal{PGF}_{\mathcal{B}}^\perp = Flat$ (\cite{EIP}, Theorem 2.14). Since $M \in \mathcal{DP}$ and $A \in Flat$, it follows that $Ext^1(M,A) = 0$. Thus $D \simeq A \oplus M$ and since $D \in \mathcal{PGF}_{\mathcal{B}}$, we have that $M \in \mathcal{PGF}_{\mathcal{B}}$.
\end{proof}
In particular, when $\mathcal{B} = Inj$  we obtain the following:

\begin{lemma}
Over any ring $R$, $\mathcal{PGF} = \mathcal{DP} \bigcap \mathcal{GF}$.
\end{lemma}



\begin{corollary}
Over any ring $R$, $\mathcal{PGF} = \mathcal{DP}$ if and only if $\mathcal{DP} \subseteq \mathcal{GF}$.
\end{corollary}

We also have the following:\\
\begin{lemma}
Over any ring $R$, $\mathcal{DP} = \mathcal{GP}$ if and only if $Flat \subseteq \mathcal{GP}^\perp$.
\end{lemma}

\begin{proof}
One implication is immediate ("$\Rightarrow$"), since, by definition, $Flat \subseteq \mathcal{DP}^\perp$. So if $\mathcal{DP} = \mathcal{GP}$ then $Flat \subseteq \mathcal{DP}^\perp = \mathcal{GP}^\perp$.

"$\Leftarrow$" We have $\mathcal{DP} \subseteq \mathcal{GP}$ over any ring. Let $M \in \mathcal{GP}$. Then there exists an exact complex of projectives $P = \ldots \rightarrow P_1 \rightarrow P_0 \rightarrow P_{-1} \rightarrow \ldots$ with $M = Z_0(P)$ and $Z_j(P) \in \mathcal{GP}$ for all $j$. Since $Flat \subseteq \mathcal{GP}^\perp$, we have $Ext^1(Z_j(P), Flat) =0$. So $Hom(P, F)$ is exact for any flat $R$-module $F$. By definition, $M \in \mathcal{DP}$.
\end{proof}

We recall that a ring $R$ is called \emph{left $n$-perfect} if every flat $R$-module has projective dimension $\le n$. Since in this case $Flat \subseteq \mathcal{GP}^\bot$, Lemma 4 gives the following:\\
\begin{corollary}
If $R$ is a left $n$-perfect ring then $\mathcal{DP} = \mathcal{GP}$.
\end{corollary}

Another consequence of Lemma 4 is the following:\\

\begin{corollary}
If $Flat \subseteq \mathcal{GP}^\perp$, then $\mathcal{PGF} = \mathcal{GP} \bigcap \mathcal{GF}$. In particular, this is the case when $R$ is a left $n$-perfect ring.
\end{corollary}

\begin{proof}
If $Flat \subseteq \mathcal{GP}^\bot$ then, by Lemma 4, we have $\mathcal{DP} = \mathcal{GP}$. Then, by Lemma 3, $\mathcal{PGF} =  \mathcal{DP} \bigcap \mathcal{GF} = \mathcal{GP} \bigcap \mathcal{GF}$.
\end{proof}

\begin{lemma}
Assume $\mathcal{B}$ is a semi-definable class of right $R$-modules. If a left $R$-module $M$ has Gorenstein $\mathcal{B}$ flat dimension $n < \infty$ then for any partial projective resolution of $M$, $0 \rightarrow G \rightarrow P_{n-1} \rightarrow \ldots \rightarrow P_0 \rightarrow M \rightarrow 0$, we have that $G$ is a Gorenstein $\mathcal{B}$ flat module.
\end{lemma}

\begin{proof}
By hypothesis there is an exact complex $0 \rightarrow G_n \rightarrow G_{n-1} \rightarrow \ldots \rightarrow G_0 \rightarrow M \rightarrow 0$ with all $G_i$ Gorenstein $\mathcal{B}$ flat modules. Since each $P_j$ is projective, there is a commutative diagram\\

\[
\begin{diagram}
\node{0}\arrow{e}\node{G}\arrow{s}\arrow{e}\node{P_{n-1}}\arrow{s}\arrow{e}\node{...}\arrow{e}\node{P_0}\arrow{s,r}{l_0}\arrow{e,t}{d_0}\node{M}\arrow{s,=}\arrow{e}\node{0}\\
\node{0}\arrow{e}\node{G_n}\arrow{e}\node{G_{n-1}}\arrow{e}\node{...}\arrow{e}\node{G_0}\arrow{e,b}{f_0}\node{M}\arrow{e}\node{0}
\end{diagram}
\]

Both rows are exact, so the mapping cone is also exact. It has an exact subcomplex $0 \rightarrow M \xrightarrow{=} M \rightarrow 0$ (see the diagram below, where the maps are: $\alpha(x,y) = (f_1(x) - l_0 (y), d_0(y))$, $\pi(x,y)=x$, $\beta(x,y) = f_1(x) - l_0(y)$, $p(x,y) = f_0(x) + y$, and $\theta(x) = x$).\\

\[
\dgARROWLENGTH=0.5em
\begin{diagram}
\node[4]{0}\arrow{e}\node{M}\arrow{s,r}{e}\arrow{e,=}\node{M}\arrow{s,r}{\theta}\arrow{e}\node{0}\\
\node{G}\arrow{s,=}\arrow{e}\node{G_n\oplus P_{n-1}}\arrow{s,=}\arrow{e}\node{...}\arrow{e,t}{\delta}\node{G_1\oplus P_0}\arrow{s,=}\arrow{e,t}{\alpha}\node{G_0\oplus M}\arrow{s,r}{\pi}\arrow{e,t}{p}\node{M}\arrow{s}\arrow{e}\node{0}\\
\node{G}\arrow{e}\node{G_n\oplus P_{n-1}}\arrow{e}\node{...}\arrow{e}\node{G_1\oplus P_0}\arrow{e,t}{\beta}\node{G_0}\arrow{e}\node{0}
\end{diagram}
\]

After factoring out the exact subcomplex $0 \rightarrow M \xrightarrow{=} M \rightarrow 0$ we obtain an exact sequence $0 \rightarrow G \rightarrow G_n \oplus P_{n-1} \rightarrow G_{n-1} \oplus P_{n-2} \rightarrow \ldots \rightarrow G_1 \oplus P_0 \rightarrow G_0 \rightarrow 0$, with all the modules $G_i$ Gorenstein $\mathcal{B}$-flat and all $P_j$ projective.
Let $L = Ker (G_{n-1} \oplus P_{n-2} \rightarrow G_{n-2} \oplus P_{n-3})$. The exact sequence $0 \rightarrow L \rightarrow G_{n-1} \oplus P_{n-2} \rightarrow \ldots \rightarrow G_1 \oplus P_0 \rightarrow G_0 \rightarrow 0$ with all $P_j$ projective and all $G_i$ Gorenstein $\mathcal{B}$ flat gives that $L$ is Gorenstein $\mathcal{B}$ flat. Then the exact sequence $0 \rightarrow G \rightarrow G_n \oplus P_{n-1} \rightarrow L \rightarrow 0$ with $G_n \oplus P_{n-1}$ and $L$ Gorenstein $\mathcal{B}$ flat modules gives that $G$ is Gorenstein $\mathcal{B}$ flat (by \cite{EIP}, Corollary 2.20, the class $\mathcal{GF}_{\mathcal{B}}$ is closed under kernels of epimorphisms).
\end{proof}

\begin{lemma}
Every module of finite flat dimension is in $\mathcal{DP}^\bot$.
\end{lemma}

\begin{proof}
- We show first that if $F \in \mathcal{DP}^\bot$ then we have $Ext^i(D,F)=0$ for all $i \ge 1$, for all $D \in \mathcal{DP}$. \\
By definition there exists an exact sequence $0 \rightarrow D' \rightarrow P \rightarrow D \rightarrow 0$ with $D' \in \mathcal{DP}$ and with $P \in Proj$. This gives an exact sequence $0 = Ext^1(D', F) \rightarrow Ext^2(D,F) \rightarrow Ext^2(P,F) =0$. This shows that $Ext^2(D,F)=0$. Similarly, $Ext^i(D,F)=0$ for all $i \ge 1$, for any $D \in \mathcal{DP}$.\\
- We prove that if $N$ has flat dimension $n < \infty$ then $Ext^1(D,N) =0$ for any $D \in \mathcal{DP}$. \\Proof by induction on $n$. If $n=0$ then $N$ is flat, so $N \in \mathcal{DP}^\bot$ by definition. For the case $n \ge 1$, consider the exact sequence $0 \rightarrow X \rightarrow F_0 \rightarrow N \rightarrow 0$ with $F_0$ flat and with $f.d. X \le n-1$. Let $D$ be any Ding projective module. By induction hypothesis, we have that $X \in \mathcal{DP}^\bot$. By the above, $Ext^i(D,X)=0$ for all $i \ge 1$. The long exact sequence $0 = Ext^1(D,F_0) \rightarrow Ext^1(D, N) \rightarrow Ext^2(D,X) =0$ shows that $Ext^1(D,N)=0$ for any $D \in \mathcal{DP}$.
\end{proof}

\begin{proposition}
Assume that $\mathcal{B}$ is a semi-definable class of right $R$-modules that contains that of injective modules. The Gorenstein $\mathcal{B}$ flat dimension of a Ding projective module is either zero or infinite.
\end{proposition}

\begin{proof}
Let $M$ be a Ding projective module which has Gorenstein $\mathcal{B}$ flat dimension $ n < \infty$. Then for any partial projective resolution of $M$, $0 \rightarrow G \rightarrow P_{n-1} \rightarrow \ldots \rightarrow P_1 \rightarrow P_0 \rightarrow M \rightarrow 0$, we have that $G \in \mathcal{GF}_{\mathcal{B}}$.\\
Since $M$ is Ding projective, each $P_i$ is projective, hence Ding projective, and since $\mathcal{DP}$ is closed under kernels of epimorphisms (by [11], Theorem 2.6), it follows that $G$ is also Ding projective. Thus $G \in \mathcal{DP} \bigcap \mathcal{GF}_{\mathcal{B}} = \mathcal{PGF}_{\mathcal{B}}$ (by Lemma 2)





So there is an exact complex $0 \rightarrow G \rightarrow P'_{n-1} \rightarrow \ldots \P'_0 \rightarrow V \rightarrow 0$ with each $P'_j$ projective and with all cycles (in particular, $V$) being $\mathcal{PGF}_{\mathcal{B}}$ modules. Since each $P_i$ is projective and $Proj \subseteq \mathcal{PGF}_{\mathcal{B}}^\bot$, we have a commutative diagram:

\[
\begin{diagram}
\node{0}\arrow{e}\node{G}\arrow{s,=}\arrow{e}\node{P'_{n-1}}\arrow{s}\arrow{e}\node{...}\arrow{e}\node{P'_0}\arrow{s}\arrow{e}\node{V}\arrow{s}\arrow{e}\node{0}\\
\node{0}\arrow{e}\node{G}\arrow{e}\node{P_{n-1}}\arrow{e}\node{...}\arrow{e}\node{P_0}\arrow{e}\node{M}\arrow{e}\node{0}
\end{diagram}
\]

Both rows are exact complexes, so the mapping cone is also exact. After factoring out the exact subcomplex $0 \rightarrow G \xrightarrow{=} G \rightarrow 0$ we obtain an exact complex $0 \rightarrow P'_{n-1} \rightarrow P_{n-1} \oplus P'_{n-2} \rightarrow \ldots \rightarrow P_1 \oplus P'_0 \rightarrow P_0 \oplus V \xrightarrow{\alpha} M \rightarrow 0$. This gives a short exact sequence $0 \rightarrow L \rightarrow P_0 \oplus V \rightarrow M \rightarrow 0$ with $L = Ker(\alpha)$ of finite projective dimension, with $P_0$ projective, $V$ a $\mathcal{PGF}_{\mathcal{B}}$ module, and with $M \in \mathcal{DP}$. By Lemma 6 we have that $Ext^1(M,L)=0$. Thus $P_0 \oplus V \simeq M \oplus L$ and therefore $M \in \mathcal{PGF}_{\mathcal{B}} \subseteq \mathcal{GF}_{\mathcal{B}}$. So $M$ has Gorenstein $\mathcal{B}$ flat dimension equal to zero.

\end{proof}

\begin{corollary} Assume that $\mathcal{B}$ is a semi-definable class of right $R$-modules that contains the class of injective modules. The following are equivalent:
\begin{enumerate}
\item $\mathcal{DP} = \mathcal{PGF}_{\mathcal{B}}$
\item Every Ding projective module has finite Gorenstein $\mathcal{B}$ flat dimension.
\end{enumerate}
\end{corollary}

\begin{proof}
(1) implies (2) is immediate since $\mathcal{PGF}_{\mathcal{B}} \subseteq \mathcal{GF}_{\mathcal{B}}$.\\
(2) $\Rightarrow$ (1) By the proof of Proposition 1, every Ding projective module of finite Gorenstein $\mathcal{B}$ flat dimension is a $\mathcal{PGF}_{\mathcal{B}}$ module.
\end{proof}

In particular when $\mathcal{B}$ is the class of injective modules we obtain:
\begin{proposition}
 The Gorenstein flat dimension of a Ding projective module is either zero or infinite.
\end{proposition}

\begin{proposition}
 The following are equivalent:
\begin{enumerate}
\item $\mathcal{DP} = \mathcal{PGF}$
\item Every Ding projective module has finite Gorenstein flat dimension.
\end{enumerate}
\end{proposition}

We recall that the left weak Gorenstein global dimension of an associative ring $R$ is defined as $l.w.Ggl.dim(R) = sup \{Gfd. M | M$ is a left $R$-module$ \}$

As an immediate consequence of Proposition 3 we obtain the following sufficient condition for the two classes of modules ($\mathcal{PGF}$ and $\mathcal{DP}$) to coincide.

\begin{proposition}
If $R$ has finite left weak Gorenstein global dimension then $\mathcal{DP} = \mathcal{PGF}$.
\end{proposition}

The following result gives another necessary and sufficient condition in order to have that $\mathcal{DP} = \mathcal{PGF}$.

\begin{lemma}
The following are equivalent:\\

1. $\mathcal{DP} \subseteq \mathcal{GF}$\\
2. $\mathcal{DP} = \mathcal{PGF}$\\
3. $\mathcal{DP} \bigcap \mathcal{PGF}^\bot \subseteq Flat$.
\end{lemma}

\begin{proof}
Statements (1) and (2) are equivalent by Corollary 2.\\
1 $\Rightarrow$ 3. If  $\mathcal{DP} \subseteq \mathcal{GF}$ then $\mathcal{DP} \bigcap \mathcal{PGF}^\bot \subseteq \mathcal{GF} \bigcap \mathcal{PGF}^\bot  = Flat$ (by \cite{saroch:18:gor}, Theorem 3.11).\\
3. $\Rightarrow$ 1. Let $M$ be a Ding projective module. Since $(\mathcal{PGF}, \mathcal{PGF}^\bot)$ is a complete cotorsion pair, there is an exact sequence $0 \rightarrow A \rightarrow B \rightarrow M \rightarrow 0$ with $A \in \mathcal{PGF}^\bot$, $B \in \mathcal{PGF} \subseteq \mathcal{DP}$. Since both $M$ and $B$ are Ding projective and the class of Ding projective modules is closed under kernels of epimorphisms (by \cite{YLL}, Theorem 2.6), it follows that $A$ is also Ding projective. Thus $A \in \mathcal{DP} \bigcap \mathcal{PGF}^\bot \subseteq Flat$, and therefore $Ext^1(M,A)=0$. It follows that $B \simeq A \oplus M$, so $M$ is a $\mathcal{PGF}$ module, hence Gorenstein flat.
\end{proof}

Theorem 1 below gives more necessary and sufficient conditions in order for the classes of Ding projectives and that of $\mathcal{PGF}$ modules to coincide.

\begin{theorem}
Let $R$ be any ring. The following are equivalent:\\
(1) $\mathcal{DP} \subseteq \mathcal{GF}$.\\
(2) $\mathcal{DP} =\mathcal{PGF}$\\
(3) For any Ding projective module $M$, its character module, $M^+$, is Gorenstein injective.\\
(4) The class $Inj ^+$ of all character modules of injective right $R$-modules, is contained in $\mathcal{DP}^\bot$.
\end{theorem}

\begin{proof}
Statements (1) and (2) are equivalent by Corollary 2.\\
2 $\Rightarrow$ 3. is immediate since any Ding projective module $M$ is Gorenstein flat in this case, so we have $M^+ \in \mathcal{GF}^+ \subseteq \mathcal{GI}$ (by \cite{holm:05:gor.dim}, Theorem 3.6).\\
3 $\Rightarrow$ 1. Let $M$ be a Ding projective module. Then there is an exact and $Hom(-, Flat)$ exact complex of projectives $P = \ldots \rightarrow P_1 \rightarrow P_0 \rightarrow P_{-1} \ldots $ such that $M = Z_0P$, and $Z_jP$ is Ding projective for all $j$. Then $P^+$ is an exact complex of injective modules, and, by (3), all cycles of $P^+$ are Gorenstein injective modules. So $Hom(I, P^+)$ is exact for any injective module $I$. Since $Hom(I, P^+) \simeq (I \otimes P)^+$ it follows that $(I \otimes P)^+ = Hom(I \otimes P, Q/Z)$ is exact. Since the abelian group $Q/Z$ is faithfully injective, it follows that the complex $I \otimes P$ is exact, for any injective right $R$-module $I$. Thus $M$ is Gorenstein flat.\\
 So statements (1), (2) and (3) are equivalent.\\
 3 $\Rightarrow$ 4. Let $M$ be a Ding projective module. For any injective module $I$ we have $Ext^1(M, I^+) \simeq Ext^1(I, M^+) = 0$ (by (3)). Thus $Inj^+ \subseteq \mathcal{DP}^\bot$.\\
 4 $\Rightarrow$ 3. Assume that $Inj^+ \subseteq \mathcal{DP}^\bot$. Let $M$ be a Ding projective module. By definition there exists an exact complex of projective modules $P$ such that $M = Z_0 P$ and $Z_jP \in \mathcal{DP}$ for all $l$. By (4) we have that $Ext^1(Z_jP, I^+) =0$ for any injective $I$. Therefore $Ext^1(I, Z_jP^+) =0$ for any injective module $I$. So  $P^+$ is an exact complex of injective modules and $Hom(I, P^+)$ is exact for any injective module $I$. It follows that $Z_jP^+$ is Gorenstein injective for all $j$. In particular, $M^+$ is Gorenstein injective.
\end{proof}

As a consequence we obtain the following:\\

\begin{theorem}
Let $R$ be a right coherent ring. Then $\mathcal{DP} = \mathcal{PGF} = \mathcal{GP}_{ac}$
\end{theorem}

\begin{proof}
 Since $R$ is coherent, $Inj ^+ \subseteq Flat \subseteq \mathcal{DP} ^ \bot$. By the above we have that $\mathcal{DP} = \mathcal{PGF}$. Also, by \cite{saroch:18:gor}, Corollary 3.5, we have that over any coherent ring $R$, $\mathcal{PGF} = \mathcal{GP}_{ac}$.
\end{proof}

Theorem 3 below is the analogue of Theorem 1 when the class of Gorenstein projective modules replaces that of Ding projectives. It gives necessary and sufficient conditions in order to have $\mathcal{PGF} = \mathcal{GP}$.\\

\begin{theorem}
Let $R$ be any ring. The following are equivalent:\\
(1) $\mathcal{GP} \subseteq \mathcal{GF}$.\\
(2) $\mathcal{GP} =\mathcal{PGF}$\\
(3) For any Gorenstein projective module $M$, its character module, $M^+$, is Gorenstein injective.\\
(4) The class $Inj ^+$ of all character modules of injective right $R$-modules, is contained in $\mathcal{GP}^\bot$.
\end{theorem}

\begin{proof}
(1) $\Rightarrow$ (2) One inclusion ($\mathcal{PGF} \subseteq \mathcal{GP} $) holds over any ring.\\
We show that if $\mathcal{GP} \subseteq \mathcal{GF}$ then we also have that $\mathcal{GP} \subseteq \mathcal{PGF}$.\\
Let $M$ be a strongly Gorenstein projective $R$-module. Then there is an exact sequence $\ldots \rightarrow P \xrightarrow{f} P \xrightarrow{f} P \rightarrow \ldots $ with $M = Ker~ f$ and $P \in Proj$. \\
Since, by hypothesis, $M$ is Gorenstein flat, we have that $Tor_i(A,M) = 0$ for any injective right $R$-module $A$. The short exact sequence $0 \rightarrow M \rightarrow P \rightarrow M \rightarrow 0$ gives an exact sequence $0 = Tor_1(A, M) \rightarrow A \otimes M \rightarrow A \otimes P \rightarrow A \otimes M \rightarrow 0$, for every injective $A_R$. Thus the sequence $0 \rightarrow A \otimes M \rightarrow A \otimes P \rightarrow A \otimes M \rightarrow 0$ is exact for any injective $A_R$, and therefore the complex $\ldots \rightarrow P \xrightarrow{f} P \xrightarrow{f} P \rightarrow \ldots $ is $Inj \otimes -$ exact. By definition $M$ is a projectively coresolved Gorenstein flat module.\\
If $M'$ is a Gorenstein projective module then $M'$ is a direct summand of a strongly Gorenstein projective module $M$ (by \cite{bennis:07:strongly}). By the above $M \in \mathcal{PGF}$, and by \cite{saroch:18:gor} Theorem 3.9, this class is closed under direct summands. Thus $M'$ is projectively coresolved Gorenstein flat.\\
(2) implies (1) is immediate since $\mathcal{PGF} \subseteq \mathcal{GF}$.\\
(1) and (3) are equivalent by \cite{IE}, Theorem 2.2.\\
(3) implies (4) is immediate, because for any injective right $R$-module $I$, for every Gorenstein projective $_RM$, we have $Ext^1(M, I^+) \simeq Ext^1(I, M^+) = 0$, since $M^+$ is Gorenstein injective.\\
(4) $\Rightarrow $ (3) Let $M$ be a Gorenstein projective $R$-module. Then $M = Z_0 (P)$ with $P$ a totally acyclic complex of projectives. Let $M_j = Z_j(P)$, for all integers $j$. Then the cycles of the acyclic complex $P^+$ are the modules $M_j ^+$. Since for any injective $R$-module $I$ we have $Ext^1(I,M_j^+) \simeq Ext^1(M_j^+, I) = 0$, it follows that the exact complex of injective modules $P^+$ stays exact when applying the functor $Hom(I,-)$, with $I_R$ injective. Thus $P^+$ is totally acyclic, so $M^+$ is Gorenstein injective.

\end{proof}

\begin{theorem}
Let $R$ be a coherent ring. Then $\mathcal{GP} = \mathcal{PGF}$ if and only if $Flat \subseteq \mathcal{GP}^\bot$. In particular, this is the case when $R$ is (right coherent and) left $n$-perfect (for some $n \ge 0$).
\end{theorem}

\begin{proof}
$\Rightarrow$ is immediate, since $Flat \subseteq \mathcal{PGF}^\bot$ and $\mathcal{PGF}^\bot = \mathcal{GP}^\bot$ in this case.\\
$\Leftarrow$ Since $R$ is coherent we have that $\mathcal{DP} = \mathcal{PGF}$ (Theorem 2). The results follows by Lemma 4.
\end{proof}

\begin{remark}
The coherence is a sufficient condition on the ring in order to have $\mathcal{DP} = \mathcal{PGF}$, but it is not a necessary condition. If $R$ has finite global dimension but it is not coherent, then we still have that $\mathcal{GP} = \mathcal{DP} = \mathcal{PGF}$.
\end{remark}

\begin{proof}
Since $R$ has finite global dimension, and so, finite weak Gorenstein global dimension, we have that $\mathcal{PGF} = \mathcal{DP}$ (by Proposition 4). Also, since $gl.d. R < \infty$ we have that $\mathcal{GP} = Proj$, and therefore $Flat \subseteq \mathcal{GP}^\bot$. By Lemma 4, $\mathcal{GP} = \mathcal{DP}$. So $\mathcal{GP} = \mathcal{DP} = \mathcal{PGF}$ in this case.
\end{proof}

\begin{example}
We showed in \cite{iacob:17} that the ring
$\displaystyle R = \left[
                     \begin{array}{ccc}
                       \mathbb{Q} & \mathbb{Q} & \mathbb{R} \\
                       0 & \mathbb{Q} & \mathbb{R} \\
                       0 & 0 & \mathbb{Q} \\
                     \end{array}
                   \right] /
                   \left[
                     \begin{array}{ccc}
                       0 & 0 & \mathbb{R} \\
                       0 & 0 & 0 \\
                       0 & 0 & 0 \\
                     \end{array}
                   \right]$

is an example of a noncoherent ring of finite global dimension. By Remark 1, $\mathcal{GP} = \mathcal{DP} = \mathcal{PGF}$ over this ring.
\end{example}

We show (Proposition 7) that if $\mathcal{B}$ is a semi-definable class of right $R$-modules that contains the class of injective right $R$-modules then every module of finite Gorenstein $\mathcal{B}$ flat dimension has a special Ding projective precover. We prove first:\\
\begin{proposition}
Assume that $\mathcal{B}$ is a semi-definable class of right $R$-modules such that $\mathcal{B} \supseteq Inj$. Then every Gorenstein $\mathcal{B}$ flat module has a special Ding projective precover.
\end{proposition}

\begin{proof}
Let $M \in \mathcal{GF}_{\mathcal{B}}$. Since $(\mathcal{PGF}_{\mathcal{B}}, \mathcal{PGF}_{\mathcal{B}}^\bot)$ is a complete hereditary cotorsion pair, there is an exact sequence $0 \rightarrow A \rightarrow B \rightarrow M \rightarrow 0$, with $B \in \mathcal{PGF}_{\mathcal{B}}$ and $A \in \mathcal{PGF}_{\mathcal{B}}^\bot$. Since $B \in \mathcal{PGF}_{\mathcal{B}} \subseteq \mathcal{GF}_{\mathcal{B}}$ and $M \in \mathcal{GF}_{\mathcal{B}}$, and $\mathcal{GF}_{\mathcal{B}}$ is a resolving class, it follows that $A \in \mathcal{GF}_{\mathcal{B}}$. Thus $A \in \mathcal{GF}_{\mathcal{B}} \bigcap \mathcal{PGF}_{\mathcal{B}}^\bot = Flat$ (by \cite{EIP}). The short exact sequence $0 \rightarrow A \rightarrow B \rightarrow M \rightarrow 0$ with $B \in \mathcal{PGF}_{\mathcal{B}} \subseteq \mathcal{DP}$ and $A \in Flat \subseteq \mathcal{DP}^\bot$ shows that $B \rightarrow M$ is a special Ding projective precover of $M$.
\end{proof}

In particular when $\mathcal{B} = Inj$ we have:\\
\begin{proposition}
Every Gorenstein flat module has a special Ding projective precover.
\end{proposition}


We recall that $\mathcal{GC}_{\mathcal{B}}$ denotes the right orthogonal class of $\mathcal{GF}_{\mathcal{B}}$ (the Gorenstein $\mathcal{B}$ cotorsion modules). The proof of Proposition 7 uses the following result:\\

\begin{lemma}
Assume that $\mathcal{GF}_{\mathcal{B}}$ is closed under extensions. If $M$ has Gorenstein $\mathcal{B}$ flat dimension $n < \infty$ then there is an exact sequence $0 \rightarrow L \rightarrow G \rightarrow M \rightarrow 0$ with $G \in \mathcal{GF}_{\mathcal{B}}$ and with $L$ of finite flat dimension.
\end{lemma}

\begin{proof}
Proof by induction on $n$. If $n=0$ then take $G=M$ and $L=0$.\\
Case $n \ge 1$.\\
We recall (from \cite{EIP} (Proposition 3.1, the proof of 1 $\Rightarrow$ 2) that if $\mathcal{GF}_{\mathcal{B}}$ is closed under extensions then we have that $\mathcal{GF}_{\mathcal{B}} \bigcap \mathcal{GC}_{\mathcal{B}} \subseteq Flat \bigcap C$ where $C$ is the class of cotorsion modules (if, moreover, $\mathcal{B}$ contains the class of injectives, then we have equality).\\

Since $\mathcal{GF}_{\mathcal{B}}$ is covering (by \cite{EIP} Proposition 2.19, because $\mathcal{GF}_{\mathcal{B}}$ is closed under extensions), there is an exact and $Hom(\mathcal{GF}_{\mathcal{B}},-)$ exact complex $ \ldots \rightarrow G_1 \xrightarrow{f_1}  G_0 \xrightarrow{f_0} M \rightarrow 0$ with each $G_i$ Gorenstein $\mathcal{B}$ flat and with each $C_i = Ker(f_i) \in \mathcal{GC}_{\mathcal{B}}$ for all $i$.
 Since the complex $0 \rightarrow C_{n-1} \rightarrow G_{n-1} \rightarrow \ldots \rightarrow G_0 \xrightarrow{f_0} M \rightarrow 0$ is exact and since $M$ has Gorenstein $\mathcal{B}$ flat dimension $n$, it follows that $C_{n-1} \in \mathcal{GF}_{\mathcal{B}}$. Since we also have $C_{n-1} \in \mathcal{GC}_{\mathcal{B}}$ it follows that $C_{n-1}$ is flat.



So there is an exact and $Hom(\mathcal{GF}_{\mathcal{B}}, -)$ exact sequence $0 \rightarrow C_{n-1} \rightarrow G_{n-1} \rightarrow \ldots \rightarrow G_1 \xrightarrow{f_1} \rightarrow G_0 \xrightarrow{f_0} M \rightarrow 0$ with $G_0$, $G_1$, $\ldots$, $G_{n-1}$ Gorenstein $\mathcal{B}$ flat, with $Ker(f_i)$ Gorenstein $\mathcal{B}$ cotorsion for all $i$, and with $C_{n-1}$ flat. \\

Let $C_i = Ker(f_i)$ for $0 \le i \le n-2$.  The short exact sequence $ 0 \rightarrow C_1 \rightarrow G_1 \xrightarrow{f_0} C_0 \rightarrow 0$ with both $C_0$ and $C_1$ Gorenstein $\mathcal{B}$ cotorsion gives that $G_1$ is also Gorenstein $\mathcal{B}$ cotorsion, and so it is both Gorenstein $\mathcal{B}$ flat and Gorenstein $\mathcal{B}$ cotorsion. Thus $G_1$ is flat. Similarly, $G_2$, $\ldots$, $G_{n-1}$ are flat modules. Thus we have an exact complex $0 \rightarrow C_{n-1} \rightarrow G_{n-1} \rightarrow \ldots \rightarrow G_1 \rightarrow C_0 \rightarrow 0$ with $C_{n-1}$ and all $G_1$, $\ldots$, $G_{n-1}$ flat modules. So there is an exact sequence $0 \rightarrow C_0 \rightarrow G_0 \rightarrow M \rightarrow 0$ with $G_0$ Gorenstein $\mathcal{B}$ flat and with $C_0$ of finite flat dimension.

\end{proof}

\begin{proposition}
Assume that $\mathcal{B}$ contains the injective modules and that $\mathcal{GF}_{\mathcal{B}}$ is closed under extensions (for example, this is the case when $\mathcal{B}$ is semi-definable and $Inj \subseteq \mathcal{B}$). Then every module of finite Gorenstein $\mathcal{B}$ flat dimension has a special Ding projective precover.
\end{proposition}

\begin{proof}
If $M$ has finite Gorenstein $\mathcal{B}$ flat dimension, then there is an exact sequence $0 \rightarrow D \rightarrow T \rightarrow M \rightarrow 0$ with $T \in \mathcal{GF_{\mathcal{B}}}$ and with $D$ of finite flat dimension.
By Proposition 6, $T$ has a special $\mathcal{DP}$ precover, so there is an exact sequence $0 \rightarrow A \rightarrow B \rightarrow T \rightarrow 0$ with $B \in \mathcal{DP}$ and with $A$ flat.

Form the pushout diagram\\

\[
\begin{diagram}
\node[2]{0}\arrow{s}\node{0}\arrow{s}\\
\node[2]{A}\arrow{s}\arrow{e,=}\node{A}\arrow{s}\\
\node{0}\arrow{e}\node{X}\arrow{s}\arrow{e}\node{B}\arrow{s}\arrow{e}\node{M}\arrow{s,=}\arrow{e}\node{0}\\
\node{0}\arrow{e}\node{D}\arrow{s}\arrow{e}\node{T}\arrow{s}\arrow{e}\node{M}\arrow{e}\node{0}\\
\node[2]{0}\node{0}
\end{diagram}
\]

So there is an exact sequence: $0 \rightarrow X \rightarrow B \rightarrow M \rightarrow 0$ with $B \in \mathcal{DP}$ and with $f.d. X < \infty$ (because $f.d. D < \infty$ and $A$ is flat), which implies $X \in \mathcal{DP}^\bot$. Thus $D \rightarrow M$ is a special $\mathcal{DP}$ precover.
\end{proof}


When $\mathcal{B} = Inj$ we obtain:\\
\begin{proposition}
Let $R$ be any ring. Every module of finite Gorenstein flat dimension has a special Ding projective precover.
\end{proposition}

Also, for $\mathcal{B}$ being the class of injective modules we have:\\

\begin{proposition}
Let $R$ be a ring such that every injective $R$-module has finite flat dimension. Then $\mathcal{GP} = \mathcal{DP} = \mathcal{PGF}$
\end{proposition}

\begin{proof}
Let $I \in Inj$. By hypothesis $f.d. I = n < \infty$. It follows that $i.d. I^+ \le n < \infty$.\\
Let $H \in \mathcal{GP}$. Then, by \cite{bennis:07:strongly}, there is a strongly Gorenstein projective module $G$ such that $G \simeq H \oplus H'$.\\
Since $G$ is strongly Gorenstein projective there exists a short exact sequence $0 \rightarrow G \rightarrow P \rightarrow G \rightarrow 0$, with $P$ projective. It follows that $Ext^i(G,-) \simeq Ext^1(G, -)$ for all $i \ge 1$.\\
Since $i.d. I^+ \le n$ we have $Ext^i(G, I^+) =0$ for all $i \ge n+1$. By the above, $Ext^i (G. I^+) = 0$ for all $i \ge 1$. It follows that $Ext^1(H,I^+)=0$ also, for any injective module $I$. Since $I^+ \in \mathcal{GP}^\bot$, for any injective module $I$, it follows (by Theorem 3) that $\mathcal{GP} = \mathcal{PGF}$.
This implies that $\mathcal{PGF} = \mathcal{GP}$ it follows that $\mathcal{GP} = \mathcal{DP} = \mathcal{PGF}$.
\end{proof}

\begin{remark}
Assume that the ring $R$ is such that there exists a nonnegative integer $n$ such that $f.d. I \le n$ for all injective $R$-modules $I$. Then, by Proposition 9, we have that $\mathcal{GP} = \mathcal{DP} = \mathcal{PGF}$. (But, by \cite{IE}, this condition is equivalent to $R$ having finite weak Gorenstein global dimension.)
\end{remark}

\begin{proposition}
Assume that $\mathcal{DP}$ is covering. Then every module in $\mathcal{DP}^\bot$ has a projective cover.
\end{proposition}

\begin{proof}
Let $M \in \mathcal{DP}^\bot$ and let $D \xrightarrow{f} M$ be a Ding projective cover. Then there is a short exact sequence $0 \rightarrow A \rightarrow D \xrightarrow{f} M \rightarrow 0$ with $A \in \mathcal{DP}^\bot$ and $D \in \mathcal{DP}$. Then $D \in \mathcal{DP}^\bot \bigcap \mathcal{DP}$, so $D$ is projective. Since any $u \in Hom(D,D)$ such that $fu = f$ must be an isomorphism, it follows that $D \rightarrow M$ is a projective cover.
\end{proof}

\begin{proposition}
If the class of Ding projectives is covering over a ring $R$, then the ring is perfect.
\end{proposition}

\begin{proof}
By Proposition 10 above, if $\mathcal{DP}$ is covering, then every flat module has a projective cover.
Let $F$ be a flat $R$-module. Consider a short exact sequence $0 \rightarrow F_0 \rightarrow P_0 \rightarrow F \rightarrow 0$ with $P_0 \rightarrow F $ a projective cover. Sine both $F$ and $P_0$ are flat modules, so is $F_0$, and therefore $F_0$ has a projective cover. Thus we can construct a minimal projective resolution of $F$, $\ldots \rightarrow P_2 \xrightarrow{d_2} P_1 \xrightarrow{d_1} P_0 \rightarrow F \rightarrow 0$ (i.e. an exact complex such that $P_0 \rightarrow F$ and $P_i \rightarrow Ker d_{i-1}$ are projective covers).\\
If $J$ is the Jacobson radical of $R$, then $d_n(P_n) \subseteq JP_{n-1}$ since $Ker (d_{n-1})$ is superfluous in $P_{n-1}$ and so $Ker (d_{n-1}) \subseteq JP_{n-1}$. So the deleted complex $\ldots \rightarrow R/J \otimes P_2 \rightarrow R/J \otimes P_1 \rightarrow R/J \otimes P_0 \rightarrow 0$ has zero differentials. Hence $Tor_1(R/J, F) \simeq R/J \otimes P_1 \simeq P_1/JP_1$. But since $F$ is flat, we have $Tor_1(R/J, F)=0$. So $P_1 = JP_1$. But then $P_1 =0$. Thus $F$ is a projective $R$ module.
Since every flat module is projective, the ring $R$ is perfect.
\end{proof}

If moreover, $R$ is coherent, then the converse is also true.

\begin{theorem}
Let $R$ be a coherent ring. Then the class of Ding projective modules is covering if and only if the ring $R$ is perfect.
\end{theorem}

\begin{proof}
"$\Rightarrow$" If $\mathcal{DP}$ is covering, then by Proposition 11, $R$ is a perfect ring.\\
"$\Leftarrow$"  Assume $R$ is perfect. Then $\mathcal{DP} = \mathcal{GP}$. Also, since $Proj = Flat$, we have that $\mathcal{PGF} = \mathcal{GF}$. Thus $\mathcal{GF} = \mathcal{PGF} \subseteq \mathcal{DP}$ if $R$ is perfect. Since $R$ is a coherent ring we have that $Inj^+ \subseteq Flat = Proj \subseteq \mathcal{GP}^\bot$. By Theorem 3, $\mathcal{GP} = \mathcal{PGF}$ in this case. So $\mathcal{PGF} = \mathcal{GF} = \mathcal{DP} = \mathcal{GP}$ in this case. In particular, since $\mathcal{GF}$ is a covering class, $\mathcal{DP}$ is covering.
\end{proof}

\end{document}